\documentclass[10pt,reqno]{amsart}

\usepackage{amsmath,amssymb,array,microtype}
\usepackage{tikz}
\usetikzlibrary{patterns}
\usepackage[cal=rsfso]{mathalfa}
\usepackage[numbers]{natbib}

\newtheorem{theorem}{Theorem}[section]

\theoremstyle{definition}
\newtheorem{definition}[theorem]{Definition}
\newtheorem{example}[theorem]{Example}

\theoremstyle{remark}
\newtheorem{remark}[theorem]{Remark}
\newenvironment{case}[1]{\smallskip\noindent {\it Case~#1.}}{}

\begin{document}

\title{Graded Betti numbers of cycle graphs and standard Young tableaux}

\author[S. Klee]{Steven Klee}
\address{Seattle University, Department of Mathematics, 901 12th Avenue, Seattle, WA 98122}
\email{klees@seattleu.edu}

\author[M.T. Stamps]{Matthew T. Stamps}
\address{KTH Royal Institute of Technology, Department of Mathematics, SE-100 44, Stockholm, Sweden}
\email{stamps@math.kth.se}

\begin{abstract}
We give a bijective proof that the Betti numbers of a minimal free resolution of the Stanley-Reisner ring of a cycle graph (viewed as a one-dimensional simplicial complex) are given by the number of standard Young tableaux of a given shape.
\end{abstract}

\date{\today}  

\maketitle


\section{Introduction}

In a recent paper, \citet{dochtermann} studied the (graded) Betti numbers $\beta_{i,j}(C_n)$ of a minimal free resolution of the Stanley-Reisner ring of the cycle graph $C_n$, viewed as a one-dimensional simplicial complex.  He showed in \cite[Theorem 4.3]{dochtermann} that the nonzero Betti numbers of the resolution are $\beta_{0,0}(C_n) = \beta_{n-2,n}(C_n) = 1$ and 
\begin{equation}\label{Betti} \beta_{j-1,j}(C_n) = \#\{\text{standard Young tableaux of shape } (j,2,1^{n-j-2})\} \end{equation} for $2 \leq j \leq n-2$. Specifically, he showed that the left- and right-hand sides of Equation (\ref{Betti}) satisfy a common recursion formula.  In this note, we offer a bijective proof of this fact that preserves a natural duality present in each of the respective objects of interest. 

\section{Preliminaries}

For the sake of brevity, we will adhere to the standard definitions and notation established in \citet{MillerSturmfels} and \citet{Stanley-cca,ec1}, and we refer to these books for any undefined terms presented throughout this paper.  We use the convention that Young tableaux are arranged in left-justified rows of weakly decreasing length in which the first (top) row is the longest and.  For a standard Young tableau $T$, we denote by $T(i,j)$ the entry in the $i^{\text{th}}$ row (from the top) and $j^{\text{th}}$ column (from the left) in $T$.  

For a simplicial complex $\Delta$ on vertex set $V$ and $W \subseteq V$, we use $\Delta[W]:=\{F \in \Delta\ : \ F \subseteq W\}$ to denote the \textit{restriction} of $\Delta$ to the vertices in $W$, $\overline{W}$ to denote the set complement of $W$ in $V$, and $C_n$ to denote the standard cycle graph on $n$ vertices, i.e., the graph on vertex set $[n]:=\{1,2,\ldots,n\}$ whose edge set consists of all pairs $\{i,j\}$ such that $i - j \equiv \pm 1 \mod n$.  

We recall Hochster's formula, which will be the main tool in our analysis. 

\begin{theorem}[Hochster's formula]\label{thm:Hochster}  Let $\Delta$ be a simplicial complex on vertex set $V$, let $\mathbf{k}$ be a field, and let $\mathbf{k}[\Delta]$ be the Stanley-Reisner ring of $\Delta$.  Then the graded Betti numbers of a minimal free resolution of $\mathbf{k}[\Delta]$ are given by 
\begin{equation}\label{Hochster}
\beta_{i,j}(\Delta) = \sum_{W \in {V \choose j}} \dim_{\mathbf{k}} \widetilde{H}_{j-i-1}(\Delta[W];\mathbf{k}).
\end{equation}
\end{theorem}

When $\Delta = C_n$, it is clear from Equation (\ref{Hochster}) that $\beta_{0,0}(C_n) = 1$ and that $\beta_{n-2,n}(C_n) = 1$ by taking $W = \emptyset$ and $W = [n]$, respectively.  Furthermore, the restriction of $C_n$ to any proper, nonempty subset of vertices can only have non-vanishing homology in dimension $0$, so the only remaining nonzero Betti numbers in the resolution of $C_n$ are those $\beta_{j-1,j}(C_n)$ with $2 \leq j \leq n-2$.  (If $j = 1$ or $j \geq n-1$, the restriction of $C_n$ to any subset of $j$ vertices is connected, and hence does not contribute to the sum in Equation (\ref{Hochster}).)  

\section{The bijection}

Our primary goal is to understand the combinatorics of $\beta_{j-1,j}(C_n)$ for $2 \leq j \leq n-2$.  By Theorem \ref{thm:Hochster}, we know every subset $W \in {[n] \choose j}$ contributes one less  than the number of connected components of $\Delta[W]$ to $\beta_{j-1,j}(C_n)$, so our initial aim will be to associate to every standard Young tableau of shape $(j,2,1^{n-j-2})$ a unique pair $(W,X)$, where $W \in {[n] \choose j}$ and $X$ represents a distinguished connected component of $C_n[W]$.  

\begin{definition} 
For $n \geq 4$ and $2 \leq j \leq n-2$, let $\mathcal{Y}(j,n)$ denote the set of standard Young tableaux of shape $(j,2,1^{n-j-2})$.  
\end{definition}

Since every standard Young tableau filled with the numbers in $[n]$ has a box labeled $1$ in its upper left corner, the number $1$ must be distinguished in terms of the restrictions $C_n[W]$ in any bijection under consideration.  At the same time, for any proper, nonempty $W \subset [n]$, the restrictions $C_n[W]$ and $C_n[\overline{W}]$ have the same number of connected components, so any proposed bijection must somehow condition on the presence/absence of $1$ in a set $W$ and the connected components of the restrictions $\Delta[W]$ or $\Delta[\overline{W}]$, based on which of these sets contains vertex $1$.

\begin{definition}
For every subset $W \subset V$, let $$m(W):= \begin{cases} \{\min(X)\ : \ X \text{ is a connected component of } C_n[W]\} & \text{if } 1 \notin W, \\ \{\min(X)\ : \ X \text{ is a connected component of } C_n[\overline{W}]\} & \text{if } 1 \in W, \end{cases}$$ and $m'(W) := m(W) \setminus \min(m(W))$. 
\end{definition}

 Since $C_n[W]$ and $C_n[\overline{W}]$ have the same number of connected components, it follows that $|m(W)|$ is equal to (and $|m'(W)|$ is one less than) the number of connected components of $\Delta[W]$.  Note that the knowledge of $1 \in W$ and $a \in m(W)$ is sufficient to determine a component of $C_n[W]$ (or $C_n[\overline{W}]$ if $1 \notin W$). 
 
 \begin{definition} 
For $n \geq 4$ and $2 \leq j \leq n-2$, let $$\mathcal{S}(j,n) = \left\{(W,a)\ : \ W \in {[n] \choose j} \text{ and } a \in m'(W)\right\}.$$ 
 \end{definition}
 
 \begin{remark}
 Observe that $m'(W)$ is implicitly required to be nonempty and hence $C_n[W]$ has at least two connected components for each $W$ under consideration here.
 \end{remark}

At this point we are ready to present our bijection between $\mathcal{Y}(j,n)$ and $\mathcal{S}(j,n)$, but before we continue, let us first turn our attention to the set $\mathcal{Y}(j,n)$ for some brief motivation:  If $T$ is an element of $\mathcal{Y}(j,n)$, then the first row of $T$ has $j$ boxes filled by unique elements of $[n]$, the first column of $T$ has $n-j$ boxes filled by unique elements of $[n]$, and, since $T$ is standard, we know the element $1$ must be located at position $T(1,1)$.  Thus, for a given $W \in {[n] \choose j}$, it is natural to associate a standard Young tableau to $W$ by first filling the first row of the table with the elements of $W$ if $1 \in W$ and otherwise filling the first column of the table by the elements of $\overline{W}$ if $1 \notin W$.  The set $W$ is not sufficient to determine a single standard Young tableau under this rule, however, because $C_n[W]$ and $C_n[\overline{W}]$ may have many connected components.  To account for the different components, we make use of the box at position $(2,2)$ of our Young diagram.  

\begin{definition}
For every $T \in \mathcal{Y}(j,n)$, let $a_T:= T(2,2)$, let $B$ be the box in $T$ that contains the number $a_T-1$, and set 
\begin{equation*}
W_T:=
\begin{cases}
\{T(1,i)\ : \ 1 \leq i \leq j\} & \text{if $B$ lies in the first row of $T$,}\\
\{T(2,2)\} \cup \{T(1,i)\ : \ 2 \leq i \leq j\} & \text{if $B$ lies in the first column of $T$}.
\end{cases}
\end{equation*}
\end{definition}

We now proceed with the main result of this paper:

\begin{theorem}\label{thm:main}
For every $n \geq 4$ and $2 \leq j \leq n-2$, the function $\phi: \mathcal{Y}(j,n) \rightarrow \mathcal{S}(j,n)$ given by $\phi(T) = (W_T,a_T)$ is a bijection.
\end{theorem}

\begin{example}
We exhibit $\phi: \mathcal{Y}(2,5) \rightarrow \mathcal{S}(2,5)$.  The subsets $W \subseteq [5]$ for which $C_5[W]$ has multiple connected components correspond to the chords of $C_5$, so $\mathcal{S}(2,5)$ consists of the following ordered pairs: $$(\{2,4\},4), \quad (\{2,5\},5), \quad (\{3,5\},5), \quad (\{1,3\},4), \quad (\{1,4\},5).$$  Moreover, $\mathcal{Y}(2,5)$ consists of the following five fillings of the shape $(2,2,1)$, which are shown below with their corresponding images under $\phi$.

\begin{center}
\scalebox{0.9}{
\begin{tabular}{p{.2\textwidth}p{.2\textwidth}p{.2\textwidth}p{.2\textwidth}p{.2\textwidth}}
\begin{tikzpicture}[scale=.5]
\draw (0,0) -- (0,3) -- (2,3) -- (2,1) -- (1,1) -- (1,0) -- (0,0);
\draw (1,1) -- (1,3);
\draw (0,2) -- (2,2);
\draw (0,1) -- (1,1);
\draw (-1,1.5) node {$T_1 = $};
\draw (.5,2.5) node {$1$};
\draw (1.5,2.5) node {$2$};
\draw (.5,1.5) node {$3$};
\draw (1.5,1.5) node {$4$};
\draw (.5,.5) node {$5$};
\end{tikzpicture}
&
\begin{tikzpicture}[scale=.5]
\draw (0,0) -- (0,3) -- (2,3) -- (2,1) -- (1,1) -- (1,0) -- (0,0);
\draw (1,1) -- (1,3);
\draw (0,2) -- (2,2);
\draw (0,1) -- (1,1);
\draw (-1,1.5) node {$T_2 = $};
\draw (.5,2.5) node {$1$};
\draw (1.5,2.5) node {$3$};
\draw (.5,1.5) node {$2$};
\draw (1.5,1.5) node {$4$};
\draw (.5,.5) node {$5$};
\end{tikzpicture}
&
\begin{tikzpicture}[scale=.5]
\draw (0,0) -- (0,3) -- (2,3) -- (2,1) -- (1,1) -- (1,0) -- (0,0);
\draw (1,1) -- (1,3);
\draw (0,2) -- (2,2);
\draw (0,1) -- (1,1);
\draw (-1,1.5) node {$T_3 = $};
\draw (.5,2.5) node {$1$};
\draw (1.5,2.5) node {$2$};
\draw (.5,1.5) node {$3$};
\draw (1.5,1.5) node {$5$};
\draw (.5,.5) node {$4$};
\end{tikzpicture}
&
\begin{tikzpicture}[scale=.5]
\draw (0,0) -- (0,3) -- (2,3) -- (2,1) -- (1,1) -- (1,0) -- (0,0);
\draw (1,1) -- (1,3);
\draw (0,2) -- (2,2);
\draw (0,1) -- (1,1);
\draw (-1,1.5) node {$T_4 = $};
\draw (.5,2.5) node {$1$};
\draw (1.5,2.5) node {$3$};
\draw (.5,1.5) node {$2$};
\draw (1.5,1.5) node {$5$};
\draw (.5,.5) node {$4$};
\end{tikzpicture}
&
\begin{tikzpicture}[scale=.5]
\draw (0,0) -- (0,3) -- (2,3) -- (2,1) -- (1,1) -- (1,0) -- (0,0);
\draw (1,1) -- (1,3);
\draw (0,2) -- (2,2);
\draw (0,1) -- (1,1);
\draw (-1,1.5) node {$T_5 = $};
\draw (.5,2.5) node {$1$};
\draw (1.5,2.5) node {$4$};
\draw (.5,1.5) node {$2$};
\draw (1.5,1.5) node {$5$};
\draw (.5,.5) node {$3$};
\end{tikzpicture}
\\
$\phi(T_1) = (\{2,4\},4)$
&
$\phi(T_2) = (\{1,3\},4)$
& 
$\phi(T_3) = (\{2,5\},5)$
&
$\phi(T_4) = (\{3,5\},5)$
&
$\phi(T_5) = (\{1,4\},5)$.
\end{tabular}
}
\end{center}

\end{example}

\begin{example}
The case that $n=6$ and $j=3$ is the first case in which we can have a restricted subcomplex with more than two connected components. If $W = \{2,4,6\}$, then $m'(W) = \{4,6\}$ and the tableaux corresponding to $(\{2,4,6\},4)$ and $(\{2,4,6\},6)$, respectively, are \medskip

\begin{center}
\scalebox{0.9}{
\begin{tabular}{m{.15\textwidth}m{.1\textwidth}m{.15\textwidth}}
\begin{tikzpicture}[scale=.5]
\draw (0,0) -- (0,3) -- (3,3) -- (3,2) -- (2,2) -- (2,1) -- (1,1) -- (1,0) -- (0,0);
\draw (0,2) -- (2,2) -- (2,3);
\draw (0,1) -- (1,1) -- (1,3);
\draw (.5,2.5) node {$1$};
\draw (1.5,2.5) node {$2$};
\draw (2.5,2.5) node {$6$};
\draw (.5,1.5) node {$3$};
\draw (1.5,1.5) node {$4$};
\draw (.5,.5) node {$5$};
\end{tikzpicture}
&
and
&
\begin{tikzpicture}[scale=.5]
\draw (0,0) -- (0,3) -- (3,3) -- (3,2) -- (2,2) -- (2,1) -- (1,1) -- (1,0) -- (0,0);
\draw (0,2) -- (2,2) -- (2,3);
\draw (0,1) -- (1,1) -- (1,3);
\draw (.5,2.5) node {$1$};
\draw (1.5,2.5) node {$2$};
\draw (2.5,2.5) node {$4$};
\draw (.5,1.5) node {$3$};
\draw (1.5,1.5) node {$6$};
\draw (.5,.5) node {$5$};
\end{tikzpicture}
.
\end{tabular}
}
\end{center}
\end{example}

\begin{proof}[Proof of Theorem \ref{thm:main}]

We begin showing that $\phi$ is injective:  Suppose that $T$ and $T'$ are tableaux for which $\phi(T) = \phi(T')$.  Let $B$ be the box in $T$ containing the number $a_T-1$ and $B'$ be the box in $T'$ containing the number $a_T'-1$.  We consider two cases based on whether or not $1 \in W_T = W_{T'}$.  

\begin{case}{1.1} Suppose $1 \in W_T = W_{T'}$.  Then the entries of the first rows of $T$ and $T'$ are the elements of $W_T = W_{T'}$.  Since $T$ and $T'$ are standard, these entries must be written in increasing order, so the first rows of $T$ and $T'$ must be the equal.  Since $a_T = a_{T'}$, we also get that $T(2,2) = T'(2,2)$. Again, since $T$ and $T'$ are standard, it follows that the remaining entries, which must all lie in the respective first columns of $T$ and $T'$, are equal.  Therefore, $T = T'$. \end{case}  

\begin{case}{1.2} Suppose $1 \notin W_T = W_{T'}$. Then the entries of the first columns of $T$ and $T'$ are the elements of the complement of $W_T = W_{T'}$ in $[n]$.  Since $T$ and $T'$ are standard, these entries must be written in increasing order, so the first columns of $T$ and $T'$ must be equal.  Since $a_T = a_{T'}$, we also get that $T(2,2) = T'(2,2)$.  Again, since $T$ and $T'$ are standard, it follows that the remaining entries, which must all lie in the respective first rows of $T$ and $T'$, are equal.  Therefore, $T = T'$.  \end{case} \smallskip

Next, we show that $\phi$ is surjective:  Let $(W,a)$ be an element in $\mathcal{S}(j,n)$.  We consider two cases based on whether or not $1 \in W$.  Recall by our construction that $1 \in W$ if and only if $a \notin W$.

\begin{case}{2.1} Suppose $1 \in W$ and consider the tableau $T$ of shape $(j,2,1^{n-j-2})$ filled in the following way: 
\begin{itemize}
\item Sort $W$ and fill it into the first row of $T$;
\item Enter $a$ in the $(2,2)$ position of $T$;
\item Sort $\overline{W}-\{a\}$ and fill it into the rest of the first column of $T$.
\end{itemize}
It is clear that this filling is well-defined and that each element of $[n]$ belongs to one of the boxes of $T$.  Let $b = T(1,2)$ and $c = T(2,1)$.  To show that $T$ is a standard filling, it suffices to prove that $a>b$ and $a>c$.  Observe that $b$ is the second-smallest element of $W$.  If $b=2$, then it is clear that $a > b$.  Otherwise $2 \notin W$, which means $\{2, \ldots, b-1\}$ is a connected component of $C_n[\overline{W}]$, which implies that $\min(m(W)) = 2$.  Thus, every element of $m'(W)$, in particular $a$, is greater than $b$, since the remaining connected components of $C_n[\overline{W}]$ are subsets of $\{b+1 \ldots, n\}$.  This proves that $a>b$.  To see that $a>c$, we recall that $a \notin W$ and, by construction, that $a$ cannot be the smallest element of $\overline{W}$.  It follows that $c$ must be the smallest element of $\overline{W}$, and hence $a>c$.  This establishes that $T$ is standard.  \end{case}

\begin{case}{2.2} Suppose $1 \notin W$ and consider the tableau $T$ of shape $(j,2,1^{n-j-2})$ filled in the following way: 
\begin{itemize}
\item Sort $\overline{W}$ and fill it into the first column of $T$;
\item Enter $a$ in the $(2,2)$ position of $T$;
\item Sort $W - \{a\}$ and fill it into the rest of the first row of $T$.
\end{itemize}
Let $b = T(1,2)$ and $c = T(2,1)$ as before.  To show that the filling of $T$ is standard, it suffices to prove that $a>b$ and $a>c$.  Observe that $c$ is the second-smallest element of $\overline{W}$.  If $c=2$, then it is clear that $a>c$.  Otherwise $2 \in W$, which means $\{2, \ldots, c-1\}$ is a connected component of $C_n[W]$, which implies that $\min(m(W)) = 2$.  Thus, every element of $m'(W)$, in particular $a$, is greater than $c$.  To see that $a > b$, we observe that $b$ is the smallest element of $W$ and hence $\min(m(W)) = b$.  Therefore, each element of $m'(W)$, in particular $a$, is greater than $b$.  This establishes that $T$ is standard. \end{case} \smallskip

In both Cases 2.1 and 2.2, it is clear from our definition that $\phi(T) = (W,a)$.
\end{proof}

\begin{remark}
We noted earlier that for any proper, nonempty subset $W \subset [n]$, the restrictions $C_n[W]$ and $C_n[\overline{W}]$ have the same number of connected components, which implies that $\beta_{j-1,j}(C_n) = \beta_{n-j-1,n-j}(C_n)$ for any $1 \leq j \leq n-1$.  This duality is expected since $\mathbf{k}[C_n]$ is known to be Gorenstein.  The duality is reflected in the combinatorics on standard Young tableau in the form of transposition.  (The transpose $T^*$ of a standard Young tableaux $T$ of shape $(j,2,1^{n-j-2})$ is a standard Young tableaux of shape $(n-j,2,1^{j-2})$.)  The bijection defined in Theorem \ref{thm:main} establishes that these two notions of duality are compatible, that is $W_{T^*} = \overline{W}_T$ and $a_{T^*} = a_T$ for the transpose $T^*$ of any tableau $T$.
\end{remark}

\section*{Acknowledgements} 
We are grateful to Alexander Engstr\"om for suggesting this problem.  This work was supported in part by NSF Grant \#1159206.

\bibliography{biblio}
\bibliographystyle{abbrvnat}

\end{document}